\documentclass[12pt]{amsart}  

\usepackage{custom_preamble} 

\begin{document}

\title{Bootstrapping partition regularity of linear systems}

\author{\tsname}
\address{\tsaddress}
\email{\tsemail}

\begin{abstract}
Suppose that $A$ is a $k \times d$ matrix of integers and write $\mathfrak{R}_A:\N \rightarrow \N\cup \{ \infty\}$ for the function taking $r$ to the largest $N$ such that there is an $r$-colouring $\mathcal{C}$ of $[N]$ with $\bigcup_{C \in \mathcal{C}}{C^d}\cap \ker A =\emptyset$.  We show that if $\mathfrak{R}_A(r)<\infty$ for all $r \in \N$ then $\mathfrak{R}_A(r) \leq \exp (\exp(r^{O_{A}(1)}))$ for all $r \geq 2$.

When the kernel of $A$ consists only of Brauer configurations -- that is vectors of the form $(y,x,x+y,\dots,x+(d-2)y)$ -- the above has been proved by Chapman and Prendiville with good bounds on the $O_A(1)$ term.
\end{abstract}

\maketitle

\section{Introduction}

Our work concerns colourings.  For a set $X$ and natural $r$ we say that $\mathcal{C}$ is an \textbf{$r$-colouring of $X$} if $\mathcal{C}$ is a cover of $X$ \emph{i.e.} $X \subset \bigcup_{C \in \mathcal{C}}{C}$, and $\mathcal{C}$ has size $r$.  In particular we shall not need our colours to be disjoint, though such colourings are included.

Suppose that $A$ is a $k\times d$ matrix of integers.  We write $\mathfrak{R}_A:\N \rightarrow \N \cup \{\infty\}$ for the function taking $r$ to the largest $N$ such that there is an $r$-colouring $\mathcal{C}$ of $[N]:=\{1,\dots,N\}$ with $\bigcup_{C \in \mathcal{C}}{C^d}\cap \ker A =\emptyset$ -- in words, such that there are no monochromatic solutions to $Ax=0$.  Note that the function $\mathfrak{R}_A$ is monotonically increasing.

Not all matrices $A$ have $\mathfrak{R}_A(r)<\infty$ for all $r \in \N$ (\emph{e.g.} if all the non-zero terms in $A$ are positive), but those that do we call \textbf{partition regular}.  There are matrices $A$ such that van der Waerden's theorem \cite[Exercise 6.3.7]{taovu::} (first proved in \cite{van::0}) is implied by the partition regularity of $A$ (see \cite[Satz I]{rad::1}), and similarly for Schur's Theorem \cite[6.12]{taovu::} (first proved in \cite{sch::4}).  Schur's theorem actually gives the stronger fact that\footnote{It is a result of Abbott and Moser \cite{abbmos::} that we cannot do much better.} $\mathfrak{R}_A(r) \leq \lfloor e r!\rfloor$, and since the celebrated work of Gowers \cite{gow::4,gow::0} we know that van der Waerden's theorem also has reasonable bounds in terms of the number of colours.

It is the purpose of this paper to use Gowers' work to show the following.
\begin{theorem}\label{thm.mn}
Suppose that $A$ is a $k\times d$ integer-valued partition regular matrix and $r\geq 2$ is natural.  Then there is some $N \leq \exp(\exp(r^{O_{A}(1)}))$ such that any $r$-colouring of $[N]$ contains a colour class $C$ and some $x \in C^d$ such that $Ax=0$.
\end{theorem}

The basic method is expounded in the model setting\footnote{The model setting has proved very fruitful for distilling the important aspects of arguments in additive combinatorics.  See the paper \cite{gre::9} and the sequel \cite{wol::3}.} of $\F_2^n$ by Shkredov in \cite[Theorem 24]{shk::7} for the purpose of illustrating how analytic techniques can be applied to colouring results.  Chapman and Prendiville in \cite{chapre::} independently discovered the argument given in \cite[Theorem 24]{shk::7} (though with some technical differences around expansion vs large Fourier coefficients) and importantly showed how it could be applied to provide good bounds in colouring problems in the integers where none were previously known.  Specifically in \cite[Theorem 1.1]{chapre::} they prove Theorem \ref{thm.mn} for Brauer configurations, meaning for a matrix $A$ whose kernel is the set of vectors of the form $(y,x,x+y,\dots,x+(d-2)y)$ for some fixed $d \geq 3$, with a doubly exponential bound on $d$ in place of the $O_A(1)$ term.  (They also show in \cite[Theorem 1.2]{chapre::} that one may replace the $O_A(1)$ term by $1+o(1)$ for Brauer configurations with $d=4$.)

It is the purpose of this note to extend the arguments of Chapman and Prendiville to partition regular linear systems.  This entails a large notational burden and as a result, while they are able to give rather good estimates for the $O_A(1)$-term when $A$ is a matrix corresponding to a Brauer configuration, we give no meaningful estimates.\footnote{Though see the remark after the proof of Theorem \ref{thm.main}.}

The above work comes on the back of a wave of investigations using analytic techniques for colouring problems.  This really took off with the paper \cite{cwasch::0} of Cwalina and Schoen, and was followed by the work of Green and collaborators \cite{grelin::,gresan::1}, then Chow, Lindqvist and Prendiville \cite{cholinpre::}, and most recently Chapman \cite{cha::6} which inspired this particular paper.

One would often like to insist that the $x$ found in Theorem \ref{thm.mn} is in a certain sense non-degenerate.  The extent to which this is possible varies, but the question has been dealt with comprehensively by Hindman and Leader in \cite{hinlea::}.  See also \cite{fragrarod::0} for a related supersaturated formulation.

\subsection*{Existing bounds on the Rado numbers $\mathfrak{R}_A(r)$}  Other than the aforementioned \cite[Theorems 1.1 \& 1.2]{chapre::} most work has focused on the case where $A$ has one row \emph{i.e.} systems with one equation which for clarity we write in the comma-delimited form $A=(a_1,\dots,a_k)$.  In this case Rado's theorem \cite[Theorem 9.5]{lanrob::} tells us that if (and only if) $A$ is partition regular then there is $\emptyset \neq I \subset [k]$ such that $\sum_{i \in I}{a_i}=0$.

Schur's theorem itself gives rather good bounds on $\mathfrak{R}_A(r)$ when $A=(1,1,-1)$, and more generally \cite[Theorem 1.3]{cwasch::0} gives singly exponential bounds when $A$ is a partition regular row.  Stronger results when the equation satisfies additional properties are given in \cite[Theorems 1.4 \& 1.5]{cwasch::0} and \cite[Theorem 4.7]{gasmortum::}.

When $A=(1,\dots,1,-1)$ the numbers $\mathfrak{R}_A(r)$ are sometimes called the generalised diagonal Schur numbers (although they are just called Schur numbers in \cite{beubre::}).  These have been computed for many values of $r$, being completely known for $r=2$ \cite[Theorem 1.3]{beubre::}, and for $r \geq 3$ the reader is directed to \cite[Table 1]{ahmsch::} for recent calculations.  Note that the bounds in Theorem \ref{thm.mn} as ineffective so, for example, when $r=2$ our result says nothing more than $\mathfrak{R}_A(2)<\infty$.

When $A$ has just one row there is a large body of work computing the exact value of $\mathfrak{R}_A(2)$ using arguments which are much more combinatorial than those in the present paper.  This work has many extensions covering things such as Rado numbers for inhomogenous equations \cite[p259]{lanrob::}; off-diagonal Rado numbers \cite[p280]{lanrob::}; and Rado numbers for disjunctive equations \cite[p293]{lanrob::}.  We restrict ourselves to recording those results which ask for bounds on $\mathfrak{R}_A(2)$ under the same hypotheses as Theorem \ref{thm.mn}.

When $A=(a_1,a_2, -a_2)$ for $a_1,a_2 \in \N$ the value of $\mathfrak{R}_A(2)$ is computed in \cite[Theorem 9.17]{lanrob::}; when $A=(a_1, a_2, -(a_1+a_2))$ for $a_1,a_2 \in \N$ the value of $\mathfrak{R}_A(2)$ is computed in \cite[Theorem 1.1]{gupthutri::}; when $A=(1, 1, a_3, -a_4)$ for $a_3,a_4 \in \N$ (where partition regularity of $A$ ensures that $a_4 \in \{1,2,a_3,a_3+1,a_3+2\}$) the value of $\mathfrak{R}_A(2)$ is computed in \cite[Theorems 3, 4 \& 8]{robmye::} for $a_4=a_3$, $a_4=a_3+1$ and $a_4=2$ respectively, with the case $a_4=a_3+2$ being trivial; and when
\begin{equation}\label{eqn.study}
A=(\overbrace{1,\dots , 1}^{n \text{ times}}, -a_{n+1},\dots , -a_k)
\end{equation}
with $a_{n+1},\dots,a_k \in \N$ and $n \geq a_{n+1}+\cdots +a_k$, the value of $\mathfrak{R}_A(2)$ is computed in \cite[Theorem 3]{sar::2}. Note that the work of \cite{lanrob::}, \cite{robmye::} and \cite{sar::2} goes further and computes $\mathfrak{R}_A(2)$ for some $A$ which are not partition regular.  (This makes sense since we may have $\mathfrak{R}_A(2)<\infty$ without $\mathfrak{R}_A(r)<\infty$ for all $r \in \N$.  See \cite[Theorem 9.2]{lanrob::} for conditions on a single row $A$ such that $\mathfrak{R}_A(2)<\infty$.)

Finally, \cite[Theorem 1.8]{cwasch::0} shows that $\mathfrak{R}_A(r)=o_{r \rightarrow \infty}(r!)$ when $A$ is as in (\ref{eqn.study}) with $n=3$, $l=2$ and $a_1,a_2=1$, beating the bound following from Schur's argument.

\subsection*{Variable conventions}  There are some conflicts between standard uses for certain symbols in different areas.  $m$, $p$ and $c$ are the parameters of an $(m,p,c)$-set in Deuber's sense (as defined in \S\ref{ssec.mpc}), so that $c$ need not be an absolute constant, and $p$ need not be prime.  $\mathcal{C}$ usually denotes a colouring and $\mathfrak{C}$ the conjugation operator (see \S\ref{ssec.gn}).  $C$ then typically denotes a colour class in $\mathcal{C}$, rather than an absolute constant.

\subsection*{Big-$O$ notation}  We use big-$O$ notation in the usual way, see \emph{e.g.} \cite[p11]{taovu::}.  The constants behind the big-$O$ and $\Omega$ expressions may depend in peculiar ways on other parameters, and we shall sometimes need some control.  We capture this in the same way as \cite[p17]{gresan::1}:

Some big-$O$ expressions will be replaced by `universal functions' of the form $f:D_1 \times \cdots \times D_k \rightarrow D_0$ where each $D_i$ is one of the sets $(0,1]$, $\N_0$, or $\N$.  If $D_i=(0,1]$ then we write $x \preceq_i y$ if and only if $y \leq x$; otherwise we write $x \preceq_i y$ if and only if $x \leq y$.  We say that $f$ is \textbf{monotone} if $f(x) \preceq_0 f(y)$ whenever $x_i \preceq_i y_i$ for all $1 \leq i \leq k$.

Note that the above is the usual order on $\N$ and $\N_0$ and the opposite of the usual order on $(0,1]$.  This reflects the fact that we shall want bounds on, say, the size of an interval which do not get too much worse as, say, a the number of colours grows -- that would be a natural number parameter -- and also as the density of some related set does not get too small -- that would be a $(0,1]$ parameter.  Our notation of monotone aligns these different notions of large and small to point in the same direction.

It is useful to note that if $f(x)=O_{a}(g(x))$ where $a \in \N_0^d$ then there is a monotone function $F:\N_0^d \rightarrow \N$ such that
\begin{equation*}
|f(x)| \leq F(a)g(x) \text{ for all }x.
\end{equation*}
This can be shown by letting $F(a)$ be the max of the constants behind the $O_{a'}$ term as $a' \preceq a$ -- a finite set.

The universal functions mapping into $\N$ or $\N_0$ will usually be denoted by $F$s with various decorations \emph{e.g.} subscripts and superscripts, while those mapping into $(0,1]$ will usually be denoted by $\eta$s with various decorations.  To avoid too many different functions, we shall often use the same functions in situations where the optimal functions are almost certainly different but where there is little cost to doing so.

\section{Setup and tools}\label{sec.deuber}

In this section we record the tools we need.  First, in \S\ref{ssec.mpc}, we explain Deuber's framework \cite{deu::} for understanding colouring problems.  This will reduce the problem to proving Theorem \ref{thm.main}.  The key tools to prove this are recorded in \S\ref{ssec.gn}.  Finally we gather a few more technical facts in \S\ref{ssec.tool}.

\subsection{Deuber's Theorem}\label{ssec.mpc}  In \cite[Satz 3.1]{deu::} Deuber proved a conjecture of Rado, and we shall use Deuber's ideas here too.  We follow the exposition and definitions of \cite{gun::}: as in \cite[Definition 2.5]{gun::}\footnote{Which Gunderson notes is slightly different to Deuber's original.}, given $m,p,c \in \N$, a set $S \subset \N$ is an \textbf{$(m,p,c)$-set} if there is some $s=(s_0,\dots,s_{m+1}) \in \N^{m+1}$ such that
\begin{equation*}
S=\bigcup_{j=0}^m{\{cs_{m-j}+i_{m-j+1}s_{m-j+1}+ \cdots + i_{m}s_{m}: -p \leq i_{m-j+1},\dots,i_{m} \leq p\}}.
\end{equation*}
For example, if $m=2$ then
\begin{equation*}
S=\{cs_2\} \cup \{cs_1+i_2s_2: -p \leq i_2 \leq p\} \cup \{cs_0+i_1s_1+i_2s_2: -p \leq i_1,i_2\leq p\},
\end{equation*}
and even more concretely, the set $\{s_1\}\cup \{s_0-s_1,s_0,s_0+s_1\}$ -- which is a three-term arithmetic progression union its common difference -- is a $(1,1,1)$-set.

We shall give a little more motivation for these sets in a moment but first we state Deuber's Theorem.
\begin{theorem}[Deuber's Theorem, {\cite[Theorem 2.8]{gun::}}]
Suppose that $m,p,c,r \in \N$.  Then there are $M,P,C \in \N$ such that any $r$-colouring of an $(M,P,C)$-set contains a monochromatic $(m,p,c)$-set.
\end{theorem}
The main result of this paper is the following.
\begin{theorem}\label{thm.main} Suppose that $m,p,c,r \in \N$.  Then there is $N \leq \exp(\exp(r^{O_{m,p,c}(1)}))$ such that any $r$-colouring of $[N]$ contains a monochromatic $(m,p,c)$-set.
\end{theorem}
Qualitatively this is a special case of Deuber's Theorem since $[N]$ is a $(1,N,N+1)$-set.

One of the reasons $(m,p,c)$-sets are important is their relationship with solutions of equations, which we now explain.  In \cite[Satz IV]{rad::1} Rado famously proved that partition regularity of a system is equivalent to something called the columns condition: we say that a $k \times d$ matrix $A$ satisfies the \textbf{columns condition} if there is a $d \times t$ matrix of rationals $\alpha$ and a partition $[d]=I_1 \sqcup \cdots \sqcup I_t$, such that writing $a_1,\dots,a_d$ for the columns of $A$ in their given order we have
\begin{equation*}
\sum_{i \in I_{j+1}}{a_i} =\sum_{i \in I_1\cup \cdots \cup I_j}{\alpha_{ij}a_i} \text{ for all }0 \leq j < t,
\end{equation*}
with the usual convention that the empty sum is $0$.  When we need to refer to a specific $\alpha$ we shall call it a \textbf{witness} for the columns condition.  It is natural to assume that $\alpha_{ij}=0$ for all $i \in I_{j+1}\cup \cdots \cup I_t$ and we shall always do so without remark.  In view of this we see that $t=1+\rk \alpha$.
\begin{theorem}[Rado's theorem, {\cite[Theorem 2.3]{gun::}}]\label{thm.rad} Suppose that $A$ is a $k \times d$ integer-valued matrix.  Then $A$ is partition regular if and only if $A$ satisfies the columns condition.
\end{theorem}
Deuber connected the columns condition to $(m,p,c)$-sets through the following.
\begin{theorem}[{\cite[Theorem 2.6(i)]{gun::}}]\label{thm.gun} Suppose that $A$ is a $k \times d$ integer-valued matrix satisfying the columns condition as witnessed by $\alpha$.  Then, writing $c$ for the least common multiple of the denominators of the rationals in $\alpha$, every $(1+\rk \alpha, \max_{i,j}{|c\alpha_{ij}|},c)$-set $S$ has some $x \in S^d$ such that $Ax=0$.
\end{theorem}
This is not \cite[Theorem 2.6(i)]{gun::} as stated, but a quick look at the proof shows that this is what is proved.

\begin{proof}[Proof of Theorem \ref{thm.mn} given Theorem \ref{thm.main}]
By Theorem \ref{thm.rad} and Theorem \ref{thm.gun}, we see that if $A$ is partition regular then there are naturals $m,p,c=O_A(1)$ such that any $(m,p,c)$-set $S$ contains some $x \in S^d$ with $Ax=0$.  By Theorem \ref{thm.main} we see that for $N \leq \exp(\exp(r^{O_{m,p,c}(1)}))=\exp(\exp(r^{O_A(1)}))$ any $r$-colouring of $[N]$ has a colour class $C$ containing a set $S$ that is an $(m,p,c)$-set, and hence there is some $x \in S^d \subset C^d$ with $Ax=0$ as required.
\end{proof}

\subsection{Gowers norms}\label{ssec.gn}  The Gowers norms are defined in \cite[Lemma 3.9]{gow::0}, though they are not given that name, and while they can be defined more generally for finite Abelian groups we shall restrict attention to cyclic groups of prime order (in line with \cite{gow::0}).  We use \cite{tao::10} as our basic reference though admittedly many of the result there are left as exercises.  The material is developed in considerable generality in \cite{gretao::7}; the generality we need is closer to that discussed in \cite[\S2]{gowwol::0}.  (Other introductions may be found in many places including \cite[\S4]{gretao:::}, \cite[\S\S2\&3]{hatlov::}, \cite[Appendix A]{wal::2}, and \cite[\S1]{man::3}.  All these, including \cite{gowwol::0}, ultimately refer to \cite{gretao::7} for details, though the paper \cite{wal::2} does expand on the details somewhat in \S4.)

For $N \in \N$ (which will be prime though need not be right now), $k \in \N$ and $f:\Z/N\Z \rightarrow \C$ we put
\begin{equation*}
\|f\|_{U^k(\Z/N\Z)}:=\left(\E_{x,h_1,\dots,h_k \in \Z/N\Z}{\prod_{\omega \in \{0,1\}^k}{\mathfrak{C}^{|\omega|}f(x+\omega\cdot h)}}\right)^{2^{-k}},
\end{equation*}
where $\mathfrak{C}$ denotes the operation of complex conjugation.

The map $\|\cdot \|_{U^k(\Z/N\Z)}$ defines a norm \cite[Exercise 1.3.19]{tao::10} for $k \geq 2$, and enjoys the nesting property $\|\cdot \|_{U^k(\Z/N\Z)} \leq \|\cdot \|_{U^{k+1}(\Z/N\Z)}$ for $k \in \N$ \cite[Exercise 1.3.19]{tao::10}.  (Proofs of these two facts are given explicitly on \cite[p466]{taovu::} and in \cite[(11.7)]{taovu::}.)

One of the reasons these norms are important is that they control counts of various linear configurations.  Specifically, suppose that $\Psi:\Z^d \rightarrow \Z^l$ is a homomorphism and $f$ is a vector of $l$ functions $\Z/N\Z \rightarrow \C$.  We define
\begin{equation}\label{eqn.dlp}
\Lambda_\Psi(f):=\E_{x \in [N]^d}{\prod_{i=1}^l{f_i(\Psi_i(x)+N\Z)}}.
\end{equation}
The following is the `generalised von Neumann Theorem' we need.  It is a special case of \cite[Theorem 4.1]{gretao:::} once the notation has been unpacked, and also of \cite[Exercise 1.3.23]{tao::10} combined with \cite[Exercise 1.3.14]{tao::10}.
\begin{theorem}\label{thm.gvn}
Suppose that $\Psi:\Z^d \rightarrow \Z^l$ is a homomorphism and for every $i \neq j$,  $(\Psi_i,\Psi_j)$ is a pair of independent vectors (\emph{i.e.} if $z\Psi_i + w\Psi_j\equiv 0$ for some $z,w \in \Z$ then $z=w=0$).  Then there are naturals $N_0(\Psi)$ and $k(\Psi)$ such that if $N \geq N_0(\Psi)$ is a prime and $f$ is a vector of $l$ functions $\Z/N\Z \rightarrow \C$ bounded by $1$ we have
\begin{equation*}
|\Lambda_\Psi(f)| \leq \inf_{1 \leq i \leq l}{\|f_i\|_{U^k(\Z/N\Z)}}.
\end{equation*}
\end{theorem}
We shall use the above to count $(m,p,c)$-sets.  This was already done by L{\^e} in \cite{Le::} for the purpose of transferring the partition regularity of Brauer configurations to the sets $\{p-1:p \text{ is prime}\}$ and $\{p+1:p \text{ is prime}\}$, itself answering a question of Li and Pan \cite{lipan::}.

We also need Gowers' inverse theorem.  The following result is what is proved in \cite[Theorem 18.1]{gow::0}, though it is not stated in precisely this way.
\begin{theorem}\label{thm.gi} There is a monotone function $F_1:\N \rightarrow \N$ such that the following holds.  Suppose that $N$ is prime, $\epsilon \leq \frac{1}{2}$ and $f:\Z/N\Z\rightarrow \C$ is bounded in magnitude by $1$ with $\|f\|_{U^k(\Z/N\Z)} \geq \epsilon$.  Then there is a partition of $[N]$ into arithmetic progressions $P_1,\dots,P_M$ of average size at least $N^{\epsilon^{F_1(k)}}$ such that
\begin{equation*}
\sum_{j=1}^M{\left|\sum_{s \in P_j}{f(s+N\Z)}\right|} \geq \epsilon^{F_1(k)}N.
\end{equation*}
\end{theorem}

\subsection{Convolution, dilation and progressions}\label{ssec.tool}

First we record notation for dilation and translation: given $x,y \in \Z$ we write $\lambda_x(y):=xy$; further, given $f:\Z \rightarrow \C$ we write $\tau_x(f)(y):=f(y+x)$. 

Suppose that $P\subset \Z$ is an arithmetic progression of odd length.  Then
\begin{equation*}
P=x_P+d_P \cdot \{-N_P,\dots,N_P\}
\end{equation*}
for some $x_P \in \Z$ called the \textbf{centre}; $d_P \in \N_0$ called the \textbf{common difference}; and $N_P \in \N_0$ called the \textbf{radius}.  Technically the common difference and radius need not be uniquely defined but this only becomes a problem for arithmetic progressions of size $1$ where the necessary adaptations of any argument are trivial and omitted for clarity.

We work with arithmetic progressions of odd length for convenience not because of any important difference.  We say that $P$ is a \textbf{centred arithmetic progression} if $x_P=0$, so in particular a centred arithmetic progression is of odd length.

For $\delta \geq 0$ we shall define
\begin{equation*}
I_\delta(P):=d_P \cdot \{-\lfloor \delta N_P\rfloor,\dots,\lfloor \delta N_P\rfloor\},
\end{equation*}
and below record some basic properties of these `fractional dilates' of progressions.  In many cases these properties are special cases of properties of Bohr sets (see \cite[\S4.4]{taovu::}).  We do not require the generality of Bohr sets here because the result of Gowers' inverse theorem (Theorem \ref{thm.gi}) is a decomposition in terms of progressions.  This has the additional benefit of meaning we do not need to deal with the problem of finding regular Bohr sets (see \cite[Lemma 4.24]{taovu::}), since all progressions are regular in a suitable sense.  This is captured in the last three properties below.
\begin{lemma}[Basic properties]\label{lem.triv}
Suppose that $P$ and $P'$ are arithmetic progressions of odd length; $c \in \N$; $x \in \Z$; and $\delta, \delta' \in (0,1]$.
\begin{enumerate}
\item \label{pt.trivsym} \emph{(Symmetry)} $I_\delta(P)$ is a centred progression of size at least $\frac{1}{3}\delta |P|$;
\item \label{pt.monpar} \emph{(Monotonicity in radius)} $I_{\delta'}(P) \subset I_{\delta}(P)$ whenever $\delta' \leq \delta$;
\item \label{pt.monprog} \emph{(Monotonicity in progression)} $I_{\delta}(P') \subset I_{\delta}(P)$ whenever $P' \subset P$;
\item \label{pt.cen} \emph{(Translation)} $x+P$ is a progression of odd length and $I_\delta(P)=I_\delta(x+P)$;
\item \label{pt.trivdil} \emph{(Dilations)} $c\cdot P$ is an arithmetic progression of odd length and $I_{\delta}(c\cdot P) = c\cdot I_\delta(P)$;
\item \label{pt.trivsa} \emph{(Sub-additivity)} $I_{\delta}(P)+I_{\delta'}(P) \subset I_{\delta+\delta'}(P)$;
\item \label{pt.comp} \emph{(Composition)} $I_{\delta}(I_{\delta'}(P))\subset I_{\delta\delta'}(P)$;
\item \label{pt.trivint}\emph{(Interiors)} there is an arithmetic progression of odd length, $\Int_\delta(P)$, such that
\begin{equation*}
\Int_\delta(P)+I_\delta(P) \subset P \text{ and } |\Int_\delta(P)| \geq (1-\delta)|P|;
\end{equation*}
\item \label{pt.trivclos}\emph{(Closures)} $P+I_\delta(P)$ is an arithmetic progression of odd length and
\begin{equation*}
|P + I_\delta(P)|\leq (1+\delta)|P|;
\end{equation*}
\item \label{pt.trivinv} \emph{(Invariance)} for all $f:\Z \rightarrow \C$ and $y \in I_\delta(P)$ we have
\begin{equation*}
|\E_{x \in P}{\tau_y(f)(x)} - \E_{x \in P}{f(x)}| \leq 2\delta \|f\|_{L_\infty}.
\end{equation*}
\end{enumerate}
\end{lemma}
\begin{proof}
(\ref{pt.trivsym}) is trivial on noting that $|I_\delta(P)| = 2\lfloor \delta N_P\rfloor +1 \geq \frac{1}{3}\delta (2N_P+1)$.  (\ref{pt.monpar}), (\ref{pt.monprog}), (\ref{pt.cen}), and (\ref{pt.trivdil}) are immediate.  (\ref{pt.trivsa}) follows since $\lfloor \delta N_P\rfloor + \lfloor \delta' N_P\rfloor \leq \lfloor (\delta+\delta') N_P\rfloor$; and (\ref{pt.comp}) since $\lfloor \delta \lfloor \delta' N\rfloor \rfloor \leq \lfloor \delta\delta'N\rfloor$. For (\ref{pt.trivint}) set
\begin{equation*}
\Int_\delta(P):=x_P+d_P\cdot \{-(N_P-\lfloor \delta N_P\rfloor),\dots,N_P-\lfloor \delta N_P\rfloor\},
\end{equation*}
so that $\Int_\delta(P)$ is an arithmetic progression of odd length, $\Int_\delta(P) + I_\delta(P) \subset P$, and
\begin{equation*}
|\Int_\delta(P)| \geq 2(N_P-\lfloor \delta N_P\rfloor) +1 = 2N_P +1 - 2\lfloor \delta N_P\rfloor \geq (1-\delta)|P|.
\end{equation*}
For (\ref{pt.trivclos}) we note that
\begin{equation*}
P+I_\delta(P) = x_P+d_P\cdot \{-N_P,\dots,N_P\} + d_P\cdot \{ -\lfloor \delta N_P\rfloor,\dots,\lfloor \delta N_P\rfloor\},
\end{equation*}
and so
\begin{equation*}
|P+I_\delta(P)| = 2(N_P+\lfloor \delta N_P\rfloor) +1 \leq (1+\delta)(2N_P+1).
\end{equation*}
For (\ref{pt.trivinv}) note if $s' \in I_\delta(P)$ then we have
\begin{align*}
& \left|\E_{s \in P}{\tau_{s'}(f)(s)} - \E_{s \in P}{f(s)}\right|\\
& \qquad = \left|\E_{s \in P}{(1-1_{\Int_\delta(P)})(s+s') f(s+s')} + \E_{s \in P}{1_{\Int_\delta(P)}(s+s')f(s+s')}\right.\\
&\qquad \qquad \left.- \E_{s \in P}{1_{\Int_\delta(P)}(s)f(s)} - \E_{s \in P}{(1-1_{\Int_\delta(P)})(s)f(s)}\right|\\
&\qquad \leq \left|\E_{s \in P}{(1-1_{\Int_\delta(P)})(s+s') f(s+s')}\right|+\left|\E_{s \in P}{(1-1_{\Int_\delta(P)})(s)f(s)}\right| \leq 2\delta \|f\|_{L_\infty}.
\end{align*}
The result is proved.
\end{proof}

\section{An example}

The notation in the final proof is quite heavy, so before turning to this we present an example case kindly suggested by one of the referees.

The arguments of \cite{chapre::} work to deal with $(1,p,c)$-sets (see \cite[Theorem 5.1]{chapre::}) and are similar to ours in terms of how these sorts of sets are dealt with.  The additional complexity we encounter is in dealing with $(m,p,c)$-sets for $m>1$.  Our approach is inductive on $m$, and our intention is that that by motivating the $m=2$ case the general argument will become clear.

We consider the problem of finding monochromatic septuples
\begin{equation}\label{eqn.config}
(x;y,x+y; z,z+x,z+y,z+x+y)
\end{equation}
in $r$-colourings of $\{1,\dots,N\}$.  Such septuples do not correspond to an $(m,p,c)$-set, but for our purposes it behaves rather like a $(2,1,1)$-set.  In fact looking for configurations of this type is a special case of Folkman's theorem \cite[Theorem 11, \S3.4]{grarotspe::0} (an explanation of the name may also be found in that reference), which was discovered independently by Folkman\footnote{Folkman's proof was unpublished, but a record of the fact he proved is found in \cite[Corollary 4]{grarot::}.}, Rado \cite{rad::4}, and Sanders \cite[Theorem 2]{san::31}.

We treat the three sets of terms in (\ref{eqn.config}) separated by semi-colons at three different scales.  In particular, we shall find arithmetic progressions $P_1$, $P_2$, and $P_3$ iteratively by using the Gowers inverse theorem (Theorem \ref{thm.gi}) to give a density increment for a colour class on a certain progression.  This increment translates to an increment to the sum over all colour classes of their maximum density on the translate of a progression.  Importantly the translate may be different for different colour classes; and the process terminates since the sum of the maximal densities is bounded above by $r$. On termination we have control of some localised Gowers norms like
\begin{equation*}
\|f\|_{U^3(P_i;P_3)}:=\E_{z \in P_3}{\|f1_{z+P_i}\|_{U^3}} \text{ for }i\in \{1,2\}.
\end{equation*}
Localised Gowers norms of this type are defined in \cite[(2.12)]{pre::0} amongst other places and the additional discussion around that definition may be of interest.

The Generalised von Neumann Theorem (Theorem \ref{thm.gvn}) ensures that control of these localised Gowers norms of a colour class $C$ on the translate $a_C+P_3$ on which $C$ has maximal density $\delta$ leads to
\begin{align*}
&\E_{x \in P_1,y \in P_2, z \in a_C+P_3}{1_C(x)1_C(y)1_C(x+y)1_C(z)1_C(x+z)1_C(y+z)1_C(x+y+z)}\\
& \qquad \qquad \approx \delta^4\E_{x \in P_1,y \in P_2}{1_C(x)1_C(y)1_C(x+y)}.
\end{align*}
Inductively we can ensure that this second term is large for some colour class $C$, and the largeness of that colour class in turn ensures that $\delta$ is large.  This gives a large count of septuples in one colour class as required.

\section{The proof}

We begin by recording some notation for linear forms associated with $(m,p,c)$-sets.  For $p,c \in \N$ and $t \in \N_0$ put
\begin{equation*}
\mathcal{D}_{p,c;t}:=\left\{ i \in \Z^{\N_0}: i_j \in  \{-p,\dots,p\} \text{ for all }j<t; i_{t} =c;\text{ and } i_{j}=0 \text{ for all }j >t\right\}.
\end{equation*}
The sets $\mathcal{D}_{p,c;t}$ as $t$ ranges $\N_0$ are disjoint.  For $m \in \N_0$ put
\begin{equation*}
\mathcal{U}_{m,p,c}:=\bigcup_{t=0}^m{\mathcal{D}_{p,c;t}}.
\end{equation*}
Then for $i \in \mathcal{U}_{m,p,c}$ write
\begin{equation}\label{eqn.lis}
L_i:\Z^{\N_0}\rightarrow \Z; s \mapsto \sum_{j:i_j \neq 0}{s_ji_j},
\end{equation}
which is well-defined since the set of $j$ such that $i_j \neq 0$ has size at most $m+1$ (and in particular is finite) for $i \in \mathcal{U}_{m,p,c}$.

The unique element of $\mathcal{D}_{p,c;t}$ with support of size $1$ is particularly important: we put
\begin{equation*}
i^*(c,t):=(\overbrace{0,\dots,0}^{t \text{ times}},c,0,\dots) \text{ and } \mathcal{D}_{p,c;t}^*:=\mathcal{D}_{p,c;t}\setminus \{i^*(c,t)\}.
\end{equation*}
Suppose that $P_0,\dots,P_m$ are arithmetic progressions.  We are interested in the count
\begin{equation}\label{eqn.qdef}
Q_{m,p,c}(A;P_0,\dots,P_m):=\E_{s_0 \in P_0,\dots,s_m \in P_m}{\prod_{i \in \mathcal{U}_{m,p,c}}{1_A(L_i(s))}},
\end{equation}
since if $P_0,\dots,P_m \subset \N$ then this quantity is non-zero only if $A$ contains an $(m,p,c)$-set.

The following is our key counting/density-increment dichotomy.
\begin{lemma}\label{lem.count}
There are monotone functions $F_2:\N^3 \rightarrow \N$ and $\eta_0:\N^3 \rightarrow (0,1]$ such that the following holds.  Suppose that $m,p,c \in \N$, $\delta\in (0,1]$ and $P_0,\dots , P_m\subset \Z$ are arithmetic progressions of odd length with
\begin{equation}\label{eqn.hyp1}
P_i \subset I_\delta(c\cdot P_m) \text{ for all }0 \leq i \leq m-1;
\end{equation}
$P'' \subset \Z$ is a centred arithmetic progression with
\begin{equation}\label{eqn.hyp2}
P'' \subset I_\delta(P_i) \text{ for all }0 \leq i \leq m;
 \end{equation}
 $P_m \subset \N$ and $A \subset \Z$ has $\alpha:=\E_{x \in c\cdot P_m}{1_A(x)}>0$.  Then at least one of the following holds:
\begin{enumerate}
\item \label{pt.smallN} $|P''| \leq \exp(\delta^{-F_2(m,p,c)})$;
\item \label{pt.ldelta} $\delta \geq \eta_0(m,p,c)$;
\item \label{pt.harder} there is an arithmetic progression of odd length, $P'''\subset \N$, with $I_1(c\cdot P''')\subset I_{F_2(m,p,c)\delta^2}(c\cdot P_m)$, such that
\begin{equation*}
|P'''| \geq |P''|^{\delta^{F_2(m,p,c)}} \text{ and } \E_{x \in c \cdot P'''}{1_A(x)} \geq \alpha+\delta.
\end{equation*}
\item \label{pt.count} or
\begin{align*}
& \left|Q_{m,p,c}(A;P_0,\dots,P_m)\right.\\
& \qquad \qquad \left. - \alpha^{|\mathcal{D}_{p,c;m}|}Q_{m-1,p,c}(A;P_0,\dots,P_{m-1})\right|\leq F_2(m,p,c)\delta^{\eta_0(m,p,c)};
\end{align*}
\end{enumerate}
\end{lemma}
The proof below is long but not at all conceptually difficult.  The length is a result of taking care with technicalities and somewhat licentious notation.  The basic idea is to use Theorem \ref{thm.gvn} to control the $Q$s by suitable uniformity norms and then Theorem \ref{thm.gi} to show that if that error is not small then there is a density increment.  There are two types of density increment, one is the expected increment resulting from large $U^k$ norm in Theorem \ref{thm.gi}.  The other results from ensuring that the density of $A$ is the same on two progressions, one of which is a small dilate of the other.  This second increment is common to arguments where groups are replaced by approximate groups -- in this case progressions -- and they perhaps originate in the work of Bourgain \cite{bou::5}.  (See \cite[(10.16)]{taovu::} and the definition of the function $G$ there.)
\begin{proof}
With $L_i$s defined as in (\ref{eqn.lis}) let $\Psi:=(L_i)_{i \in \mathcal{U}_{m,p,c}}$ so that $\Psi:\Z^{m+1} \rightarrow \Z^{\mathcal{U}_{m,p,c}}$ is a homomorphism.  Every $(L_i,L_j)$ with $i \neq j$ is a pair of independent vectors, so by Theorem \ref{thm.gvn} applied to $\Psi$ there is $k=k(\Psi)=O_{m,p,c}(1)$ such that if $N \geq N_0(\Psi)$ is prime and $h$ is a vector of functions $\Z/N\Z \rightarrow \C$ (indexed by $\mathcal{U}_{m,p,c}$) then
\begin{equation}\label{eqn.cnt}
\left|\E_{x \in [N]^{m+1}}{\prod_{i \in \mathcal{U}_{m,p,c}}{h_i(L_i(x)+N\Z)}}\right|\leq \inf_{i \in \mathcal{U}_{m,p,c}}{\|h_i\|_{U^{k}(\Z/N\Z)}}.
\end{equation}
Note that this will be applied with a vector $h$ to be determined, but which will be a combination of translates of the set $A$ suitably restricted, and also translates of the balanced function of $A$.  The particular choice is made in (\ref{eqn.hdef}) (which itself depends on (\ref{eqn.gdef}) and (\ref{eqn.fdef})).

As remarked in the subsection on big-$O$ notation, since $m$, $p$ and $c$ are in $\N$ we see that there is a monotone function $F:\N^3 \rightarrow \N$ such that $F(m,p,c) \geq N_0(\Psi)$.  (We shall take $F_2 \geq F$, but there are other functions later which will determine exactly what $F_2$ needs to be.)

Since $P''$ is a centred arithmetic progression there are natural numbers $d''$ and $N''$ such that $P''=d''\cdot \{-N'',\dots,N''\}$.  By Bertrand's postulate there is a prime $N$ such that
\begin{equation*}
\max\{(mp+c)N'',F(m,p,c)\} <N =O_{m,p,c}(N'').
\end{equation*}
The reason for this choice will become clear just before (\ref{eqn.newrhs}) below.  Before that we record the following claim.
\begin{claim*}
There is $\delta'=O_{m,p,c}(\delta)$ such that if $s_0 \in P_0, \dots ,s_{m-1} \in P_{m-1}$, $x \in d''\cdot \{-N,\dots,N\}$ and $i \in \mathcal{D}_{p,c;m}$ then
\begin{equation*}
L_i(s)+x-cs_m \in c\cdot I_{\delta'}(P_m) 
\end{equation*}
and if additionally $s_m \in \Int_{\delta'}(P_m)$ then
\begin{equation*}
L_i(s)+x \in c \cdot P_m.
\end{equation*}
\end{claim*}
\begin{proof}
Write $l=\left\lceil \frac{N}{N''}\right\rceil$ so by (\ref{eqn.hyp1}), (\ref{eqn.hyp2}) and Lemma \ref{lem.triv} we have
\begin{align*}
d''\cdot \{-N,\dots,N\} &\subset {\overbrace{(d''\cdot \{-N'',\dots,N''\})+\cdots + (d''\cdot  \{-N'',\dots,N''\})}^{l\text{ times}}}\\ & = lP'' \subset I_{l\delta}(P_0) \subset I_{l\delta}(I_\delta(c\cdot P_m)) \subset I_{l\delta^2}(c\cdot P_m).
\end{align*}
If follows that if $i \in \mathcal{D}_{p,c;m}$, $s_0\in P_0,\dots,s_{m-1} \in P_{m-1}$, and $x \in d''\cdot \{-N,\dots,N\}$, then by (\ref{eqn.hyp1}) and Lemma \ref{lem.triv} we have
\begin{equation}\label{eqn.kul}
i_0s_0 + \cdots + i_{m-1}s_{m-1} + x \in mpI_\delta(c\cdot P_m)+I_{l\delta^2}(c\cdot P_m)\subset I_{(mp+l\delta)\delta}(c\cdot P_m).
\end{equation}
Let $\delta':=(mp+l\delta)\delta=O_{m,p,c}(\delta)$.  If $i \in \mathcal{D}_{p,c;m}$ we have $i_m=c$ and then by (\ref{eqn.kul}) and Lemma \ref{lem.triv} we have
\begin{equation*}
L_i(s)+x  -cs_m= i_0s_0 + \cdots + i_{m-1}s_{m-1} +x \in I_{\delta'}(c\cdot P_m) = c\cdot I_{\delta'}(P_m)
\end{equation*}
giving the first conclusion.  Finally, if $s_m \in \Int_{\delta'}(P_m)$ then by Lemma \ref{lem.triv} again
\begin{equation*}
L_i(s)+x  \in cs_m + c\cdot I_{\delta'}(P_m) \subset c\cdot (\Int_{\delta'}(P_m) +I_{\delta'}(P_m) ) \subset c\cdot P_m.
\end{equation*}
The claim is proved.
\end{proof}
Let $\epsilon>0$ be a further constant to be optimised later and suppose (using $\tau$ for translation as defined in \S\ref{ssec.tool}) that for some $j \in \mathcal{D}_{p,c;m}$ we have
\begin{equation}\label{eqn.small}
\E_{s_0 \in P_0,\dots,s_m \in P_m}{1_{\Int_{\delta'}(P_m)}(s_m)\left|\E_{x \in d''\cdot [N]}{\tau_{L_j(s)}(1_A - \alpha 1_{c \cdot P_m})(x)}\right|} > \epsilon.
\end{equation}
By interchanging order of summation we have
\begin{align}
\label{eqn.split} & \left|\E_{s_0 \in P_0,\dots,s_m \in P_m}{1_{\Int_{\delta'}(P_m)}(s_m)\E_{x \in d''\cdot [N]}{\tau_{L_j(s)}(1_A - \alpha 1_{c \cdot P_m})(x)}}\right|\\
\nonumber& \qquad  = \left|\E_{s_0 \in P_0,\dots,s_{m-1} \in P_{m-1},x \in d''\cdot [N]}{}\right.\\
\nonumber & \qquad \qquad \qquad \quad\left(\E_{s_m \in P_m}{1_{\Int_{\delta'}(P_m)}(s_m)\tau_{L_j(s)+x-cs_m}1_A (cs_m)}\right.\\
\nonumber & \qquad \qquad \qquad  \qquad \qquad  \qquad \left.\left.-\alpha\E_{s_m \in P_m}{1_{\Int_{\delta'}(P_m)}(s_m)\tau_{L_j(s)+x-cs_m}(1_{c\cdot P_m})(cs_m )}\right)\right|.
\end{align}
Suppose that $s_0 \in P_0,\dots,s_{m-1} \in P_{m-1}$ and $x \in d''\cdot [N]$.  The claim and Lemma \ref{lem.triv} tell us that
\begin{align*}
\E_{s_m \in P_m}{1_{\Int_{\delta'}(P_m)}(s_m)\tau_{L_j(s)+x-cs_m}1_A (cs_m)}  & \leq\E_{s_m \in P_m}{\tau_{L_j(s)+x-cs_m}1_A (cs_m)}\\
& \leq  \E_{s_m \in P_m}{1_A(cs_m)}+ 2\delta' = \alpha+2\delta';
\end{align*}
and
\begin{align*}
\E_{s_m \in P_m}{1_{\Int_{\delta'}(P_m)}(s_m)\tau_{L_j(s)+x-cs_m}1_A (cs_m)}  & \geq\E_{s_m \in P_m}{\tau_{L_j(s)+x-cs_m}1_A (cs_m)} -\delta'\\
& \geq  \E_{s_m \in P_m}{1_A(c\cdot s_m)}-3\delta' = \alpha-3\delta'.
\end{align*}
We conclude that
\begin{equation*}
|\E_{s_m \in P_m}{1_{\Int_{\delta'}(P_m)}(s_m)\tau_{L_j(s)+x-cs_m}1_{A} (cs_m)}- \alpha| = O(\delta')=O_{m,p,c}(\delta),
\end{equation*}
and similarly
\begin{equation*}
|\E_{s_m \in P_m}{1_{\Int_{\delta'}(P_m)}(s_m)\tau_{L_j(s)+x-cs_m}1_{c\cdot P_m} (cs_m)}- 1| = O(\delta')=O_{m,p,c}(\delta)
\end{equation*}
for all $s_0 \in P_0,\dots,s_{m-1} \in P_{m-1}$ and $x \in d''\cdot [N]$.  It follows from these and (\ref{eqn.split}) that
\begin{equation*}
\left|\E_{s_0 \in P_0,\dots,s_m \in P_m}{1_{\Int_{\delta'}(P_m)}(s_m)\E_{x \in d''\cdot [N]}{\tau_{L_j(s)}(1_A - \alpha 1_{c \cdot P_m})(x)}}\right|=O_{m,p,c}(\delta).
\end{equation*}
Combining this with (\ref{eqn.small}) and averaging we see that there are elements $s_0 \in P_0,\dots,s_{m-1} \in P_{m-1},s_{m} \in \Int_{\delta'}(P_m)$ such that
\begin{equation}\label{eqn.cd}
\E_{x \in d''\cdot [N]}{\tau_{L_j(s)}(1_A - \alpha 1_{c \cdot P_m})(x)}> \frac{1}{2}\epsilon - O_{m,p,c}(\delta).
\end{equation}
Since $s_m \in \Int_{\delta'}(P_m)$ the claim tells us
\begin{equation*}
\tau_{L_j(s)}(1_{c\cdot P_m})(x) = 1_{c\cdot P_m}(L_j(s)+x) =1 \text{ for all }x \in d''\cdot \{-N,\dots,N\},
\end{equation*}
and so $L_j(s)-d''\cdot \{-N,\dots,N\} \subset c\cdot P_m \subset \N$.  It follows from this that $c$ divides $L_j(s)$ and of course $c$ divides $d''$ (since $P'' \subset I_1(P_0) \subset I_1(c\cdot P_m)=c\cdot I_1(P_m)$ by Lemma \ref{lem.triv}).  From (\ref{eqn.cd}) we then have
\begin{equation*}
\E_{x \in c\cdot (c^{-1}L_j(s)-(c^{-1}d'')\cdot [N])}{1_A(x)} >\alpha + \frac{1}{2}\epsilon - O_{m,p,c}(\delta).
\end{equation*}
We can take $\epsilon \in [\delta,O_{m,p,c}(\delta)]$ such that the right hand side is at least $\alpha+\delta$, and we are in case (\ref{pt.harder}) of the lemma (with $P''':=c^{-1}L_j(s)-(c^{-1}d'') \cdot [N]$ so $|P'''|=N \geq N''=|P''|$ where $|P'''|$ is odd since it is prime, and 
\begin{equation*}
I_1(c\cdot P''') =I_1(d''\cdot [N])\subset I_{O_{m,p,c}(1)}(P'') \subset I_{O_{m,p,c}(\delta^2)}(c\cdot P_m)
\end{equation*}
by Lemma \ref{lem.triv}.  Again by Lemma \ref{lem.triv} and the discussion in the section on big-$O$ notation there is a monotone function $F':\N^3 \rightarrow \N$ such that $I_1(c\cdot P''') \subset I_{F(m,p,c)\delta^2}(c\cdot P_m)$.  And, again, we shall take $F_2 \geq F$.)\footnote{This is much stronger than the conclusion offered in case (\ref{pt.harder}) but this is because this is the easy density increment mentioned at the end of the discussion before the proof of this lemma.}

In view of the above we may suppose that (\ref{eqn.small}) does not happen for any $j \in \mathcal{D}_{p,c;m}$; we are in the main case.

We split the integrand in (\ref{eqn.qdef}) into two factors
\begin{equation*}
\prod_{i \in \mathcal{U}_{m,p,c}}{1_A(L_i(s))}=\left(\prod_{t=0}^{m-1}{\prod_{i \in \mathcal{D}_{p,c;t}}{1_A(L_i(s))}}\right)\cdot \left(\prod_{i \in \mathcal{D}_{p,c;m}}{1_A(L_i(s))}\right).
\end{equation*}
The first term on the right is independent of $s_m$; we decompose the second through an arbitrary fixed total order on $\mathcal{D}_{p,c;m}^*$.  For $i,j \in \mathcal{D}_{p,c;m}^*$ put
\begin{equation}\label{eqn.fdef}
f_{j,i}:=\begin{cases} 1_A & \text{ if }i<j\\
1_A - \alpha 1_{c\cdot P_m} & \text{ if } i=j\\
\alpha 1_{c\cdot P_m} & \text{ if }i>j\end{cases};
\end{equation}
so for all $x \in \Z^{\mathcal{D}_{p,c;m}}$ we have
\begin{align*}
& \sum_{j \in \mathcal{D}_{p,c;m}^*}{1_A(x_{i^*(c,m)})\prod_{i \in \mathcal{D}_{p,c;m}^*}{f_{j,i}(x_i)}}\\ & \qquad \qquad = \prod_{i \in \mathcal{D}_{p,c;m}}{1_A(x_i)} -\alpha^{|\mathcal{D}_{p,c;m}^*|}1_A(x_{i^*(c,m)})\prod_{i \in \mathcal{D}_{p,c;m}^*}{1_{c\cdot P_m}(x_i)}.
\end{align*}
It follows that
\begin{align}
\label{eqn.i}&Q_{m,p,c}(A;P_0,\dots,P_m)\\
\nonumber & \qquad -  \alpha^{|\mathcal{D}_{p,c;m}^*|}\E_{s_0 \in P_0,\dots,s_{m-1} \in P_{m-1}}{\left(\prod_{t=0}^{m-1}{\prod_{i \in \mathcal{D}_{p,c;t}}{1_A(L_i(s))}}\right)}\\
\nonumber & \qquad \qquad\qquad\qquad\qquad\qquad\qquad\qquad\times\left(\E_{s_m \in P_m}{1_A(L_{i^*(c,m)}(s))\prod_{i \in \mathcal{D}_{p,c;m}^*}{1_{c\cdot P_m}(L_i(s))}}\right)\\
\nonumber &\qquad \qquad  = \sum_{j \in \mathcal{D}_{p,c;m}^*}{\E_{s_0 \in P_0,\dots,s_{m} \in P_{m}}{\left(\prod_{t=0}^{m-1}{\prod_{i \in \mathcal{D}_{p,c;t}}{1_A(L_i(s))}}\right)}}\\
\nonumber &\qquad \qquad \qquad\qquad\qquad\qquad\qquad\qquad\qquad \times \left(1_A(L_{i^*(c,m)}(s))\prod_{i \in \mathcal{D}_{p,c;m}^*}{f_{j,i}(L_i(s))}\right).
\end{align}
On the other hand for all $i \in \mathcal{D}_{p,c;m}$ and $s \in P_0 \times \cdots \times P_{m-1} \times \Int_{\delta'}(P_m)$ we have from the claim (with $x=0$) that $L_i(s) \in c\cdot P_m$.  Hence for all $s_0 \in P_0,\dots,s_{m-1} \in P_{m-1}$ we have
\begin{align*}
& \E_{s_m \in P_m}{1_A(L_{i^*(c,m)}(s))\prod_{i \in \mathcal{D}_{p,c;m}^*}{1_{c\cdot P_m}(L_i(s))}}\\
&\qquad \qquad \geq \E_{s_m \in P_m}{1_A(cs_m)1_{c\cdot \Int_{\delta'}(P_m)}(cs_m)\prod_{i \in \mathcal{D}_{p,c;m}^*}{1_{c\cdot P_m}(L_i(s))}}\\
& \qquad \qquad =  \E_{s_m \in P_m}{1_A(cs_m)1_{c\cdot \Int_{\delta'}(P_m)}(cs_m)}\geq  \alpha-O_{m,p,c}(\delta).
\end{align*}
On the other hand
\begin{equation*}
 \E_{s_m \in P_m}{1_A(L_{i^*(c,m)}(s))\prod_{i \in \mathcal{D}_{p,c;m}^*}{1_{c\cdot P_m}(L_i(s))}}\leq  \E_{s_m \in P_m}{1_A(L_{i^*(c,m)}(s))} = \alpha,
\end{equation*}
and so
\begin{equation*}
 \E_{s_m \in P_m}{1_A(L_{i^*(c,m)}(s))\prod_{i \in \mathcal{D}_{p,c;m}^*}{1_{c\cdot P_m}(L_i(s))}}= \alpha+O_{m,p,c}(\delta).
\end{equation*}
Moreover,
\begin{equation*}
Q_{m-1,p,c}(A;P_0,\dots,P_{m-1})=\E_{s_0 \in P_0,\dots,s_{m-1} \in P_{m-1}}{\left(\prod_{t=0}^{m-1}{\prod_{i \in \mathcal{D}_{p,c;t}}{1_A(L_i(s))}}\right)}
\end{equation*}
since the sets $\mathcal{D}_{p,c;t}$ are disjoint over $0 \leq t <m$.  We conclude that the left hand side of (\ref{eqn.i}) is equal to
\begin{equation*}
Q_{m,p,c}(A;P_0,\dots,P_m) - \alpha^{|\mathcal{D}_{p,c;m}|}Q_{m-1,p,c}(A;P_0,\dots,P_m)+O_{m,p,c}(\delta).
\end{equation*}
To estimate the right hand side of (\ref{eqn.i}) first note (by (\ref{eqn.hyp2}) and Lemma \ref{lem.triv}) that for $j \in \mathcal{D}_{p,c;m}^*$ the summand equals
\begin{align}
\label{eqn.innerexpect}&\E_{s_0 \in P_0,\dots,s_{m} \in P_{m}}{\E_{s_0',\dots,s_m' \in d''\cdot [N'']}{\left(\prod_{t=0}^{m-1}{\prod_{i \in \mathcal{D}_{p,c;t}}{1_A(L_i(s+s'))}}\right)}}\\
\nonumber &\qquad \qquad\qquad  \qquad\qquad \qquad\qquad  \times \left(1_A(L_{i^*(c,m)}(s+s'))\prod_{i \in \mathcal{D}_{p,c;m}^*}{f_{j,i}(L_i(s+s'))}\right)\\
\nonumber &  \qquad \qquad\qquad  \qquad\qquad \qquad\qquad\qquad \qquad\qquad+ O_{m}(\delta).
\end{align}
We shall look at the inner expectation of the first term here for which it will be useful to introduce some more notation.  We shall use $\lambda$ for dilation in the way defined in \S\ref{ssec.tool}, and then for $s \in P_0 \times \cdots \times P_m$ and $y \in \Z^{m+1}$ put
\begin{equation}\label{eqn.gdef}
g_i^{(s)}(y)=\begin{cases}
\tau_{L_i(s)}(1_A) \circ \lambda_{d''}(y) & \text{ if }i \in \mathcal{U}_{m-1,p,c} \text{ or }i=i^*(c,m)\\
\tau_{L_i(s)}(f_{j,i})\circ \lambda_{d''}(y) & \text{ if } i \in \mathcal{D}_{p,c;m}^*.
\end{cases}
\end{equation}
With this notation, the inner expectation in (\ref{eqn.innerexpect}) equals
\begin{align*}
& \E_{y_0,\dots,y_m \in [N'']}{\prod_{i \in \mathcal{U}_{m,p,c}}{g_i^{(s)}(L_i(y))}}\\
& \qquad = \frac{1}{(N'')^{m+1}} \cdot \sum_{y \in \Z^{m+1}}{\left(\prod_{t=0}^m{1_{[N'']}(y_t)}\right)\left(\prod_{t=0}^m{\prod_{i \in \mathcal{D}_{p,c;t}^*}{g_i^{(s)}(L_i(y))}}\right)\left(\prod_{t=0}^m{g_{i^*(c,t)}^{(s)}(L_{i^*(c,t)}(y))}\right)}\\
& \qquad = \frac{1}{(N'')^{m+1}} \cdot \sum_{y \in \Z^{m+1}}{\left(\prod_{t=0}^m{\prod_{i \in \mathcal{D}_{p,c;t}^*}{g_i^{(s)}(L_i(y))}}\right)\left(\prod_{t=0}^m{g_{i^*(c,t)}^{(s)}|_{c\cdot [N'']}(L_{i^*(c,t)}(y))}\right)},
\end{align*}
since $L_{i^*(c,t)}(y)=cy_t$ for all $0 \leq t \leq m$.  The notation is potentially a little confusing here: $g_{i^*(c,t)}^{(s)}|_{c\cdot [N'']}$ denotes the function $g_{i^*(c,t)}^{(s)}$ restricted to the set $c\cdot [N'']$.

For $x \in [N]$ write
\begin{equation}\label{eqn.hdef}
h_i^{(s)}(x+N\Z)=\begin{cases}
g_i^{(s)}(x) & \text{ if }i \in \bigcup_{t=0}^m{\mathcal{D}_{p,c;t}^*} \\
g_{i}^{(s)}|_{c \cdot [N'']}(x) & \text{ if }i \in \{i^*(c,t):0 \leq t \leq m\}
\end{cases}.
\end{equation}
In view of this definition, for $x \in [N]^{m+1}$, the product
\begin{equation*}
\prod_{i \in \mathcal{U}_{m,p,c}}{h_i^{(s)}(L_i(x)+N\Z)}
\end{equation*}
is non-zero only if $x \in ((c\cdot [N'']+N\Z)^{m+1})\cap ([N]^{m+1})$.  This set equals $(c\cdot [N''])^{m+1}$ since $N >cN''$.  Now, if $x \in (c\cdot [N''])^{m+1}$ then $L_i(x) \in [N]$ since $N > (mp+c)N''$ and so $h_i^{(s)}(x+N\Z)=g_i^{(s)}(x)$.  It follows that
\begin{equation}\label{eqn.newrhs}
\sum_{x \in [N]^{m+1}}{\prod_{i \in \mathcal{U}_{m,p,c}}{h_i^{(s)}(L_i(x)+N\Z)}} = (N'')^{m+1}\cdot \E_{x_0,\dots,x_m \in [N'']}{\prod_{i \in \mathcal{U}_{m,p,c}}{g_i^{(s)}(L_i(x))}}.
\end{equation}
Apply (\ref{eqn.cnt}) to the above and conclude that the right hand side of (\ref{eqn.i}) is at most
\begin{align*}
&\sum_{j \in \mathcal{D}_{p,c;m}^*}{\E_{s_0 \in P_0,\dots,s_{m} \in P_{m}}{\left(\frac{N}{N''}\right)^{m+1}\left\|h_j^{(s)}\right\|_{U^k(\Z/N\Z)}}}+ O_{m,p,c}(\delta).
\end{align*}
Let $\eta:=(4\sqrt{\epsilon})^{\frac{1}{F_1(k)}} + \sqrt{\epsilon} + \delta'$ (where $F_1$ is as in Theorem \ref{thm.gi}) and suppose
\begin{equation}\label{eqn.earear}
\sum_{j \in \mathcal{D}_{p,c;m}^*}{\E_{s_0 \in P_0,\dots,s_m \in P_m}{\left\|h_j^{(s)}\right\|_{U^k(\Z/N\Z)}}} < \eta |\mathcal{D}_{p,c;m}^*|.
\end{equation}
Then it follows that
\begin{equation*}
\left|Q_{m,p,c}(A;P_0,\dots,P_m) - \alpha^{|\mathcal{D}_{p,c;m}|}Q_{m-1,p,c}(A;P_0,\dots,P_m)\right| = O_{m,p,c}(\delta) + O_{m,p,c}(\eta),
\end{equation*}
and we will find ourselves in case (\ref{pt.count}) in view of the definition of $\eta$ and choice of $\epsilon$ earlier.  On the other hand suppose that (\ref{eqn.earear}) does not hold, so that by averaging there is some $j \in \mathcal{D}_{p,c;m}^*$ such that
\begin{equation}\label{eqn.upit}
\E_{s_0 \in P_0,\dots,s_m \in P_m}{\left\|h_j^{(s)}\right\|_{U^k(\Z/N\Z)}} \geq \eta.
\end{equation}
Since (\ref{eqn.small}) does not happen, writing
\begin{equation*}
\mathcal{S}:=\{s \in P_0 \times \cdots \times P_m: s_m \in \Int_{\delta'}(P_m) \text{ and } \left|\E_{x \in d''\cdot [N]}{\tau_{L_j(s)}(1_A - \alpha 1_{c \cdot P_m})(x)}\right|>\sqrt{\epsilon} \}
\end{equation*}
we have
\begin{equation*}
\E_{s_0 \in P_0,\dots,s_m \in P_m}{1_{\mathcal{S}}(s)\sqrt{\epsilon}} \leq \epsilon.
\end{equation*}
 By (\ref{eqn.upit}), Lemma \ref{lem.triv}, the value of $\eta$ and the triangle inequality we see that
\begin{equation*}
\E_{s_0 \in P_0,\dots,s_m \in P_m}{1_{\Int_{\delta'}(P_m)}(s_m)1_{\mathcal{S}^c}(s)\left\|h_j^{(s)}\right\|_{U^k(\Z/N\Z)}} \geq \eta - \sqrt{\epsilon} - \delta' \geq (4\sqrt{\epsilon})^{\frac{1}{F_1(k)}}.
\end{equation*}
By averaging there is some $s \in (P_0\times \cdots \times P_{m-1}\times \Int_{\delta'}(P_m))\setminus \mathcal{S}$ such that
\begin{equation*}
\left\|h_j^{(s)}\right\|_{U^k(\Z/ N\Z)}\geq (4\sqrt{\epsilon})^{\frac{1}{F_1(k)}}.
\end{equation*}
By Theorem \ref{thm.gi} (applicable since $4\sqrt{\epsilon} \leq 2^{-F_1(k)}$ as otherwise we are in case (\ref{pt.ldelta}) of the Lemma since $\epsilon \geq \delta$) there is a partition of $[N]$ into arithmetic progressions $Q_1,\dots,Q_M$ of average size at least $N^{\epsilon^{F_1(k)}}$ such that
\begin{equation}\label{eqn.yy}
\sum_{l=1}^M{\left|\sum_{x \in Q_l}{\tau_{L_j(s)}(f_{j,j}) \circ \lambda_{d''}(x)}\right|} =\sum_{l=1}^M{\left|\sum_{x \in Q_l}{g_j^{(s)}(x)}\right|} =\sum_{l=1}^M{\left|\sum_{x \in Q_l}{h_j^{(s)}(x+N\Z)}\right|} \geq 4\sqrt{\epsilon}N.
\end{equation}
Of course
\begin{align*}
\sum_{l=1}^M{\sum_{x \in Q_l}{\tau_{L_j(s)}(f_{j,j})\circ \lambda_{d''}(x)}} & = \sum_{x \in d''\cdot [N]}{\tau_{L_j(s)}(f_{j,j})(x)}\\
& =  \sum_{x \in d''\cdot [N]}{\tau_{L_j(s)}(1_A - \alpha 1_{c \cdot P_m})(x)},
\end{align*}
and so (since $s \not \in \mathcal{S}$)
\begin{equation*}
\left|\sum_{l=1}^M{\sum_{x \in Q_l}{\tau_{L_j(s)}(f_{j,j})\circ \lambda_{d''}(x)}}\right| \leq \sqrt{\epsilon }N.
\end{equation*}
Moreover, since the average size of $Q_l$ is at least $N^{\epsilon^{F_1(k)}}$, we have
\begin{equation*}
\sum_{\substack{1 \leq l \leq M\\|Q_l| \leq N^{\epsilon^{F_1(k)/2}}}}{\left|\sum_{x \in Q_l}{\tau_{L_j(s)}(f_{j,j}) \circ \lambda_{d''}(x)}\right|} \leq N^{\epsilon^{F_1(k)/2}}M \leq N^{-\epsilon^{F_1(k)/2}}N \leq \sqrt{\epsilon }N,
\end{equation*}
since we may assume $N^{-\epsilon^{F_1(k)/2}}\leq \sqrt{\epsilon}$ (or else we are in case (\ref{pt.smallN}) of the Lemma since $\epsilon \geq \delta$ and $N \geq N''=|P''|$).  We may assume all the $Q_l$s are of odd size by removing at most one point from each at a cost of at most $M$ in (\ref{eqn.yy}).  (Again $M\leq \sqrt{\epsilon}N$ or else we are in case (\ref{pt.smallN}) of the Lemma.)  It follows by the triangle inequality and averaging that there is some $1 \leq l \leq M$ with
\begin{equation*}
\sum_{x \in Q_l}{\tau_{L_j(s)}(f_{j,j}) \circ \lambda_{d''}(x)} \geq \frac{1}{2}\sqrt{\epsilon }|Q_l| \text{ and } |Q_l| > N^{\epsilon^{F_1(k)}/2}.
\end{equation*}
Rewriting the first expression we get that
\begin{align*}
\E_{x \in d''\cdot Q_l}{1_A(L_j(s) + x)} & \geq \alpha \E_{x \in d''\cdot Q_l}{1_{c\cdot P_m}(L_j(s)+x)} + \frac{1}{2}\sqrt{\epsilon} = \alpha + \frac{1}{2}\sqrt{\epsilon}
\end{align*}
by the claim.  First, $\frac{1}{2}\sqrt{\epsilon }\geq 2\delta$ (or else we are in case (\ref{pt.ldelta}) of the Lemma in view of the fact that $\epsilon \geq \delta$).  Secondly, as noted after (\ref{eqn.cd}), $L_j(s)-d''\cdot \{-N,\dots,N\} \subset c\cdot P_m \subset \N$, $c$ divides $L_j(s)$ and $c$ divides $d''$, and so
\begin{equation*}
\E_{x \in c^{-1}L_j(s)-(c^{-1}d'')\cdot Q_l}{1_A( x)} \geq \alpha + \delta,
\end{equation*}
and putting $P''':=c^{-1}L_j(s)-(c^{-1}d'')\cdot Q_l \subset \N$ we are in case (\ref{pt.harder}) and hence are done.  (Indeed, $|P'''|=|Q_l| \geq N^{\epsilon^{F_1(k)}/2} \geq N^{\delta^{F_1(k)}/2}$, and $|Q_l|$ is odd.  As before
\begin{equation*}
I_1(c\cdot P''') =I_1(d''\cdot [N])\subset I_{O_{m,p,c}(1)}(P'') \subset I_{O_{m,p,c}(\delta^2)}(c\cdot P_m)
\end{equation*}
by Lemma \ref{lem.triv}.  Again by Lemma \ref{lem.triv} and the discussion in the section on big-$O$ notation there is a monotone function $F':\N^3\rightarrow \N$ such that $I_1(c\cdot P''') \subset I_{F(m,p,c)\delta^2}(c\cdot P_m)$.  And, again, we shall take $F_2 \geq F$.)

The lemma is proved.
\end{proof}

Our main result is the following theorem.
\begin{theorem}\label{thm.key}
Suppose that $m,p,c,r \in \N$, $\delta \in (0,1]$; $P \subset \N$ is an arithmetic progression of odd length $N$; and $\mathcal{C}$ is an $r$-colouring of $I_1(P)\cap \N$.  Then at least one of the following holds.
\begin{enumerate}
\item \label{cs.n3b} $N \leq \exp(\exp(\delta^{-O_{m,p,c}(1)}))$;
\item \label{cs.d2b} $\delta \geq (2r)^{-O_{m,p,c}(1)}$;
\item \label{cs.conc} there are progressions $P_0,\dots, P_m \subset I_\delta(c\cdot P)\cap \N$ of odd length with
\begin{equation*}
P_i \subset I_1(P_{i+1}) \text{ for all }0 \leq i \leq m-1 \text{ and } |P_0| \geq N^{\exp(-\delta^{-O_{m,p,c}(1)})},
\end{equation*}
and some $C \in \mathcal{C}$ such that
\begin{equation*}
Q_{m,p,c}(C;P_0,\dots,P_m)\geq (2r)^{-O_{m,p,c}(1)}.
\end{equation*}
\end{enumerate}
\end{theorem}
We shall proceed by a double induction.  The outer induction will be on $m$ and the inner is a density increment argument.
\begin{lemma}[Iteration Lemma]\label{lem.iil}
Suppose that Theorem \ref{thm.key} holds for some $m \in \N_0$, \emph{i.e.}

\begin{center}\fbox{\begin{minipage}{35em}There are monotone functions $F^{(m)}:\N^2\times (0,1] \rightarrow \N$, $\eta^{(m)}_0:\N^3\rightarrow (0,1]$ and $\eta^{(m)}_1:\N^2 \times (0,1] \rightarrow (0,1]$ such that the following holds.  For any $p,c,r \in \N$, $\delta \in (0,1]$, $P\subset \N$ an arithmetic progressions of odd length $N$, and $r$-colouring $\mathcal{C}$ of $I_1(P)\cap \N$ at least one of the following holds.
\begin{enumerate}
\item \label{cs.n3} $N \leq F^{(m)}(p,c,\delta)$;
\item \label{cs.d2} $\delta \geq \eta_0^{(m)}(p,c,r)$;
\item there are progressions $P_0,\dots, P_m \subset I_\delta(c\cdot P)\cap \N$ of odd length with
\begin{equation*}
P_i \subset I_1(P_{i+1}) \text{ for all }0 \leq i \leq m-1 \text{ and } |P_0| \geq N^{\eta_1^{(m)}(p,c,\delta)},
\end{equation*}
and some $C \in \mathcal{C}$ such that
\begin{equation*}
Q_{m,p,c}(C;P_0,\dots,P_m)\geq \eta_0^{(m)}(p,c,r).
\end{equation*}
\end{enumerate}\end{minipage}}\end{center}

Then for any $p,c,r \in \N$, $\delta\in (0,1]$, $P\subset \N$ an arithmetic progressions of odd length $N$, and $r$-colouring $\mathcal{C}$ of $I_1(P)\cap \N$ at least one of the following holds.
\begin{enumerate}
\item \label{cs.n4}
\begin{equation*}
N \leq \min\left\{F^{(m)}(p,c,\delta),\exp(\delta^{-O_{m,p,c}(1)}\eta_1^{(m)}(p,c,\delta)^{-1})\right\}
\end{equation*}
\item \label{cs.d5}
\begin{equation*}
\delta \geq\left(\frac{\eta_0^{(m)}(p,c,r)}{2r}\right)^{O_{m,p,c}(1)};
\end{equation*}
\item \label{cs.inc} or there is a progression $P'''$ of odd length with
\begin{equation*}
I_1(P''') \subset I_1(P) \text{ and } |P'''| \geq N^{\delta^{O_{m,p,c}(1)} \eta_1^{(m)}(p,c,\delta)}
\end{equation*}
such that
\begin{equation*}
\sum_{C \in \mathcal{C}}{\max_{y: y+ P''' \subset \N}{\E_{x \in c\cdot (y+P''')}{1_C(x)}}} \geq \sum_{C \in \mathcal{C}}{\max_{y: y+ P \subset \N}{\E_{x \in c\cdot (y+P)}{1_C(x)}}}+ \frac{1}{2}\delta;
\end{equation*}
\item \label{cs.ct} there are progressions $P_0,\dots, P_{m+1} \subset I_\delta(c\cdot P)\cap \N$ of odd length with
\begin{equation*}
P_i \subset I_1(P_{i+1}) \text{ for all }0 \leq i \leq m \text{ and } |P_0| \geq N^{\eta_1^{(m)}(p,c,\delta)},
\end{equation*}
and some $C \in \mathcal{C}$ such that
\begin{equation*}
Q_{m+1,p,c}(C;P_0,\dots,P_{m+1})\geq \left(\frac{\eta_0^{(m)}(p,c,r)}{2}\right)^{O_{m,p,c}(1)}.
\end{equation*}
\end{enumerate}
\end{lemma}
\begin{proof}
Apply the content of the box to $P$ to get that either we are in case (\ref{cs.n3}) or (\ref{cs.d2}) of the hypothesis and so in case (\ref{cs.n4}) and (\ref{cs.d5}) respectively of the present lemma, or else there are progressions $P_0,\dots,P_m \subset I_\delta(c\cdot P)\cap \N$ with $P_i \subset I_1(P_{i+1})$ for all $0 \leq i \leq m-1$, and some $C \in \mathcal{C}$ with
\begin{equation*}
Q_{m,p,c}(C;P_0,\dots,P_m) \geq \eta_0^{(m)}(p,c,r) \text{ and }|P_0| \geq N^{\eta_1^{(m)}(p,c,\delta)}.
\end{equation*}
Let $y_0 \in \Z$ be such that $y_0+P \subset \N$ and
\begin{equation}\label{eqn.max}
\E_{x \in c\cdot (y_0+P)}{1_C(x)}=\max_{y: y+ P \subset \N}{\E_{x \in c\cdot (y+P)}{1_C(x)}},
\end{equation}
and let $P_{m+1}=y_0+P$.  Then $P_0,\dots,P_{m+1}$ are arithmetic progressions of odd length. Furthermore, by Lemma \ref{lem.triv} (and assuming we are not in case (\ref{cs.d5})) we have
\begin{equation*}
P_i \subset I_\delta(c\cdot P_m) \subset I_\delta(c\cdot I_\delta(c\cdot P)) \subset I_\delta(c\cdot P_{m+1}) \text{ for all } 0 \leq i \leq m-1,
\end{equation*}
and
\begin{equation*}
P_m \subset I_\delta(c\cdot P) = I_\delta(c\cdot P_{m+1}).
\end{equation*}
It follows that we can apply Lemma \ref{lem.count} with parameters $m+1, p, c \in \N$ and $\delta\in (0,1]$, set $C$, and odd length arithmetic progressions $P_0,\dots,P_{m+1} \subset \Z$, and
\begin{equation*}
P'':=I_\delta(P_0) \subset I_\delta(P_i) \text{ for all }0 \leq i \leq m+1.
\end{equation*}
We have four cases.
\begin{enumerate}
\item \emph{(Case ({\ref{pt.smallN}}))} Then
\begin{equation*}
\delta N^{\eta_1^{(m)}(p,c,\delta)} \leq |I_\delta(P_0)| = |P''| \leq \exp(\delta^{-F_2(m+1,p,c)}),
\end{equation*}
and we are in case (\ref{cs.n4}) of this lemma. 
\item \emph{(Case ({\ref{pt.ldelta}}))} 
\begin{equation*}
\delta\geq \eta_0(m+1,p,c)
\end{equation*}
and we are in case (\ref{cs.d5}) of the lemma.
\item \emph{(Case ({\ref{pt.harder}})} Then there is an arithmetic progression $P''' \subset \N$ of odd length with $I_1(c\cdot P''') \subset I_{F_2(m+1,p,c)\delta^2}(c\cdot P)$ such that
\begin{equation*}
|P'''| \geq N^{\eta_1^{(m)}(p,c,\delta)\delta^{O_{m,p,c}(1)}} \text{ and } \E_{x \in c\cdot P'''}{1_C(x)} \geq \E_{x \in c\cdot P_{m+1}}{1_C(x)} + \delta.
\end{equation*}
(Then either we are in case (\ref{cs.d5}) of the lemma or else $I_1(P''') \subset I_1(P)$.)  In view of the choice of $y_0$ (\ref{eqn.max}) and the definition of $P_{m+1}$ the second expression tells us that
\begin{equation*}
\max_{w:w+P''' \subset \N}{\E_{z \in c\cdot (w+P''')}{1_{C}(x)}}\geq \E_{x \in c\cdot P'''}{1_C(x)} \geq \max_{y:y+P\subset \N}{\E_{x \in c\cdot (y+P)}{1_C(x)}} + \delta.
\end{equation*}
For (the other) $C' \in \mathcal{C}$ and $y \in \Z$ such that $y+P \subset \N$ we have, by Lemma \ref{lem.triv}, that
\begin{align*}
\max_{w:w+P''' \subset \N}{\E_{z \in c\cdot (w+P''')}{1_{C'}(z)}}& \geq \E_{x \in c\cdot P}{\E_{z \in c\cdot (x+y+P''')}{1_{C'}(z)}}\\ & =\E_{u \in c\cdot (y+P)}{\E_{z \in c\cdot P'''}{\tau_z(1_{C'})(u)}}\\ &\geq \E_{x \in c\cdot (y+P)}{1_{C'}(x)} - F_2(m+1,p,c)\delta^2.
\end{align*}
Taking the maximum over $y$ such that $y+P \subset \N$ and summing it follows that
\begin{align*}
& \sum_{C' \in \mathcal{C}}{ \max_{w:w+P''' \subset \N}{\E_{z \in c\cdot (w+P''')}{1_{C'}(z)}}}\\
&\qquad \qquad \geq\sum_{C' \in \mathcal{C}}{ \max_{w:w+P \subset \N}{\E_{z \in c\cdot (w+P)}{1_{C'}(z)}}} +\delta - (r-1)F_2(m+1),p,c)\delta^2.
\end{align*}
So we are either in case (\ref{cs.d5}) of the lemma or (\ref{cs.inc}) of the lemma.
\item \emph{(Case ({\ref{pt.count}}))} Then 
\begin{align}
\label{eqn.ss}& \left|Q_{m+1,p,c}(C;P_0,\dots,P_{m+1})\right.\\
\nonumber & \qquad \qquad \left. - \alpha^{|\mathcal{D}_{m+1,p,c}|}Q_{m,p,c}(C;P_0,\dots,P_{m})\right|\leq F_2(m+1,p,c)\delta^{\eta_0(m+1,p,c)},
\end{align}
where $\alpha:=\E_{x \in c\cdot P_{m+1}}{1_C(x)}$. First, suppose that
\begin{equation}\label{eqn.falssup}
 \max_{y:y+P_m \subset \N}{\E_{x \in c\cdot (y+P_m)}{1_C(x)}}>\alpha + ((r-1)F_2(m+1,p,c)+1)\delta
\end{equation}
and so by (\ref{eqn.max}) and the definition of $P_{m+1}$ we have
\begin{align*}
 \max_{y:y+P_m \subset \N}{\E_{x \in c\cdot (y+P_m)}{1_C(x)}}&>\max_{y:y+P \subset \N}{\E_{x \in y+P}{1_C(x)}}\\ &\qquad \qquad+((r-1)F_2(m+1,p,c)+1)\delta.
\end{align*}
For the other $C' \in \mathcal{C}$, we use that $P_m \subset I_\delta(c\cdot P_{m+1})=I_\delta(c\cdot P)$, and Lemma \ref{lem.triv} to give that for any $y+P \subset \N$ we have
\begin{align*}
\max_{w:w+P_m \subset \N}{\E_{z \in c\cdot (w+P_m)}{1_{C'}(z)}}& \geq \E_{x \in c\cdot (y+P)}{\E_{z \in c\cdot (x+y+P_m)}{1_{C'}(z)}}\\ & =\E_{u \in c\cdot (y+P)}{\E_{z \in c\cdot P'''}{\tau_z(1_{C'})(u)}}\\ &\geq \E_{x \in c\cdot (y+P)}{1_{C'}(x)} - F_2(m+1,p,c)\delta.
\end{align*}
Taking the maximum over $y$ such that $y+P \subset \N$ and summing it follows that
\begin{align*}
& \sum_{C' \in \mathcal{C}}{ \max_{w:w+P_m \subset \N}{\E_{z \in c\cdot (w+P_m)}{1_{C'}(z)}}}\\
&\qquad \qquad \geq\sum_{C' \in \mathcal{C}}{ \max_{w:w+P \subset \N}{\E_{z \in c\cdot (w+P)}{1_{C'}(z)}}} +\delta.
\end{align*}
We are in case (\ref{cs.inc}).  (We should also note that we are in case (\ref{cs.d5}) of the lemma or else $I_1(P_m) \subset I_1(P)$.)  We conclude that (\ref{eqn.falssup}) does not hold and so
\begin{align*}
& \alpha+ ((r-1)F_2(m+1,p,c)+1)\delta\\  &\qquad \qquad \geq  \max_{y:y+P_m \subset \N}{\E_{x \in c\cdot (y+P_m)}{1_C(x)}}\\
 & \qquad \qquad \geq \E_{x \in c\cdot P_m}{1_C(x)} \geq Q_{m,p,c}(C;P_0,\dots,P_m) \geq \eta_0^{(m)}(p,c,r).
\end{align*}
Either
\begin{equation*}
\delta \geq \frac{\eta_0^{(m)}(p,c,r)}{2((r-1)F_2((m+1),p,c)+1)},
\end{equation*}
and we are in case (\ref{cs.d5}) of the lemma; or from (\ref{eqn.ss}) we have
\begin{equation*}
F_2(m+1,p,c)\delta^{\eta_0(m+1,p,c)} \geq \frac{1}{2}\left(\frac{1}{2}\eta_0^{(m)}(p,c,r)\right)^{|\mathcal{D}_{m+1,p,c}|+1}
\end{equation*}
and we are in case (\ref{cs.d5}) of the lemma; or we are in case (\ref{cs.ct}) of the lemma.
\end{enumerate}
The lemma is proved.
\end{proof}

\begin{proof}[Proof of Theorem \ref{thm.key}]
We proceed by induction on $m$ to show that the content of the box in Lemma \ref{lem.iil} holds.  This gives the theorem.  The result holds for $m=0$ and it is convenient to use that as the base case.  To see this take $P_0:=I_\delta(c\cdot P)\cap \N$.  If $\delta N <2$ then we are in case (\ref{cs.n3}) of the box.  If not then $|P_0| =\Omega(\delta N)$, and again we are either in case (\ref{cs.n3}), or $|P_0|$ satisfies the required lower bound.  Finally, we are in case (\ref{cs.d2}) or else $I_\delta(c^2\cdot P) \subset I_1(P)$ and
\begin{equation*}
\sum_{C \in \mathcal{C}}{Q_{0,p,c}(C;P_0)} = \sum_{C \in \mathcal{C}}{\E_{s_0 \in P_0}{1_C(cs_0)}} \geq  \E_{s_0 \in P_0}{1_{I_1(P)\cap \N}(cs_0)} =1.
\end{equation*}
The result follows by averaging.  Now, suppose we have proved that the content of the box holds for some $m$.

We proceed iteratively defining progressions $P^{(0)},P^{(1)},\dots$ with $I_{1}( P^{(j+1)}) \subset I_1(P^{(j)})$ for all $j \geq 0$.  Begin with $P^{(0)}:=I_1(P)\cap \N$ and define
\begin{equation*}
\mu_j:=\sum_{C \in \mathcal{C}}\max_{y:y+P^{(j)} \subset \N}{\E_{z \in y+P^{(j)}}{1_C(z)}}.
\end{equation*}
By hypothesis we have $\mu_0 \geq 1$ and we also have $\mu_j \leq r$ for all $j$.  At stage $j \in \N_0$ we apply Lemma \ref{lem.iil} to $P^{(j)}$ and unless we are in case (\ref{cs.inc}) we terminate.  If we are in case (\ref{cs.inc}) then we let $P^{(j+1)}$ be the progression given, which has
\begin{equation*}
I_1(P^{(j+1)}) \subset I_1(P^{(j)}), |P^{(j+1)}| \geq |P^{(j)}|^{\delta^{O_{m,p,c}(1)}\eta_1^{(m)}(p,c,\delta)} \text{ and }\mu_{j+1} \geq \mu_j + \frac{1}{2}\delta.
\end{equation*}
In view of the last fact this iteration can proceed for at most $2\delta^{-1}$ steps before terminating.  When it terminates we have either
\begin{equation*}
N^{\delta^{-O_{m,p,c}(\delta^{-1})}(\eta_1^{(m)}(p,c,\delta))^{2\delta^{-1}}} \leq \min \left\{F^{(m)}(p,c,\delta),\exp(\delta^{-O_{m,p,c}(1)}\eta_1^{(m)}(p,c,\delta)^{-1})\right\};
\end{equation*}
or
\begin{equation*}
\delta \geq \left(\frac{\eta_0^{(m)}(p,c,r)}{2r}\right)^{O_{m,p,c}(1)};
\end{equation*}
or there is some $C \in \mathcal{C}$ such that
\begin{equation*}
Q_{m+1,p,c}(C;P_0,\dots,P_{m+1}) \geq \left(\frac{\eta_0^{(m)}(p,c,r)}{2}\right)^{O_{m,p,c}(1)} \text{ and } |P_0| \geq N^{\delta^{-O_{m,p,c}(\delta^{-1})}(\eta_1^{(m)}(p,c,\delta))^{2\delta^{-1}}}.
\end{equation*}
It follows that we can take
\begin{equation*}
F^{(m+1)}(p,c,\delta) \leq (2F^{(m)}(p,c,\delta))^{\eta_1^{(m)}(p,c,\delta)^{-O_{m,p,c}(\delta^{-1})}};
\end{equation*}
\begin{equation*}
\eta_1^{(m+1)}(p,c,\delta) \geq \left(\frac{\eta_1^{(m)}(p,c,\delta)}{2}\right)^{O_{m,p,c}(\delta^{-1})};
\end{equation*}
and
\begin{equation*}
\eta_0^{(m+1)}(p,c,r) \geq \left(\frac{\eta_0^{(m)}(p,c,r)}{2r}\right)^{O_{m,p,c}(1)}.
\end{equation*}
These recursions give the claimed bounds.
\end{proof}

\begin{proof}[Proof of Theorem \ref{thm.main}]
We apply Theorem \ref{thm.key} with $P=[N]$ (or $P=[N-1]$ if $N$ is even) with $\frac{1}{2c} \geq \delta = r^{-O_{m,p,c}(1)}$ such that case (\ref{cs.d2b}) never holds.  If the colouring contains no $(m,p,c)$-set then $Q_{m,p,c}(C;P_0,\dots,P_m)=0$ for any $P_0,\dots,P_m$ of form described in case (\ref{cs.conc}) and so that does not happen.  We conclude that $N$ is bounded in a way that yields the result.
\end{proof}

As a final remark, although we have made no effort to track the $m$, $p$, and $c$ dependencies they should also not be too bad given the known bounds in Theorem \ref{thm.gvn} and Theorem \ref{thm.gi}.

\section*{Acknowledgements} The author should like to thank David Conlon and Julia Wolf for useful conversations, Jonathan Chapman and Sean Prendiville for providing an early copy of \cite{chapre::}, and the referees for careful reading of the paper including the example of how the main argument works.

\bibliographystyle{halpha}

\bibliography{references}

\end{document}